\documentclass[12pt, reqno]{amsart}
\usepackage{graphicx}
\usepackage{subfigure}
\usepackage{latexsym}
\usepackage{amsmath}
\usepackage{amssymb}
\usepackage{mathrsfs}
\usepackage{color}

 \newtheorem{theorem}{Theorem}[section]
 \newtheorem{Def}[theorem]{Definition}
 \newtheorem{Prop}[theorem]{Proposition}
 \newtheorem{Lem}[theorem]{Lemma}
 \newtheorem{Cor}[theorem]{Corollary}

 \newtheorem{Exa}[theorem]{Example}

 \numberwithin{equation}{section}

\begin{document}

\title{Topological structure of fractal squares}

\author{Ka-Sing Lau} \address{Department of Mathematics, The Chinese University of Hong Kong,
Hong Kong} \email{kslau@math.cuhk.edu.hk}
\author{Jun Jason Luo} \address{Department of Mathematics,  Shantou
University, Shantou 515063, China} \email{luojun2011@yahoo.com.cn}
\author{Hui Rao} \address{Department of Mathematics, Hua Zhong Normal University, Wuhan 430079,
China} \email{hrao@mail.ccnu.edu.cn}




\thanks{The research is supported by the HKRGC grant,  the Focus Investment Scheme of CUHK, and  the NNSF of China (no.
10871065).  Luo is also supported by STU Scientific Research
Foundation for Talents (no. NTF12016).  Rao is supported by the NNSF
of China (no. 11171128)}



\date{\today}

\begin{abstract}
Given an integer $n\geq 2$ and a digit set ${\mathcal D}\subsetneq
\{0,1,\dots,n-1\}^2$, there is a self-similar set $F \subset {\Bbb R}^2$ satisfying the set
equation: $F=(F+{\mathcal D})/n$. We call such $F$ a fractal square.
By studying a  periodic extension  $H= F+ {\mathbb
Z}^2$, we classify $F$ into three types according to their
topological properties. We also provide some simple criteria for
such classification.
\end{abstract}

\maketitle

\begin{section}{\bf Introduction}

 For $n\geq 2$, let ${\mathcal{D}}\subset
\{0,1,\dots,n-1\}^2$ and call it a \emph{digit set}.  We assume that
$1<\#{\mathcal D}<n^2$ to exclude the trivial case. Let $F \subset {\Bbb R}^2$ be the
unique non-empty compact set  satisfying the set equation
\cite{Fa03}
\begin{equation} \label{eq1.1}
F= (F+{\mathcal{D}})/n.
\end{equation}
We shall call $F$  a {\it fractal square}. A familiar example of a
fractal square is the Sierpinski carpet.  Let $I=[0,1]^2$ be the
unit square. We define $F_1= (I+{\mathcal{D}})/n$, and recurrently,
$F_{k+1}= (F_k+{\mathcal{D}})/n$ for $k\geq 1$. Then $F_k$ is a
union of squares of size $1/n^k$ (called them the $k$-{\it cells}).
Clearly $F_{k+1}\subset F_k$ and $F=\bigcap_{k=1}^{\infty}F_k$.

\medskip

The topological  structure of self-similar/self-affine sets,
including connectedness, local connectedness, disk-likeness,
 is an  important topic in fractal geometry.   In [6],  Hata
 first gave a criterion for connectedness  of  self-similar sets. Subsequently, there are many
works devoted to study such topological properties
(\cite{BaGe}, \cite{BaWa01}, \cite{DeLa11}, \cite{KiLa00},
\cite{LeLa07}, \cite{LeLu}, \cite{LeLu2}, \cite{LuRaTa02},
\cite{NgTa04}, \cite{NgTa05}). Recently, Taylor et al \cite{Tay1,
Tay2} considered the connectedness properties of the Sierpinski
relatives  with rotations and reflections.  Xi and
Xiong \cite{XiXi10} showed that for a fractal square, it is  totally
disconnected  if and only if {\it the number of cells in the
connected components in each iteration is uniformly bounded}.
Moreover, Roinestad \cite{Ro10} proved that  a fractal square $F$ is
totally disconnected if and only if {\it for some $k\geq 1$,\
$I\setminus F_k$ contains  a path that can reach the opposite side
of the square}.
\medskip

Our aim in the paper is to  provide a complete characterization on
the topological structure of the fractal square through the
connected components. We introduce periodic extensions of $F$ and
$F_k$ by defining
$$
H=F+{\mathbb{Z}}^2 \quad\text{and}\quad  H_k=F_k+{\mathbb{Z}}^2
$$
and denote their complements by $H^c$ and $H_k^c$, respectively.
Then we can classify $F$ into three classes topologically. We
summarize the results in the following theorem. For brevity, we will
use  {\it component} to  mean {\it connected component}, and a {\it
non-trivial component} means it has more than one point.

\medskip

\begin{theorem} Let $F$ be a fractal square as in  (\ref{eq1.1}). Then $F$ satisfies either

\ {\rm (i)} $H^c$ has a bounded  component, which is also equivalent to: $F$ contains a non-trivial component that is not a line segment;
or

{\rm (ii)} $H^c$ has an unbounded  component, then $F$ is either
totally disconnected or all non-trivial components of $F$ are
parallel line segments.
\end{theorem}

\medskip

 It is easy to see that the Sierpinski carpet is of type (i). The two cases in type (ii) are not so obvious. The reader can refer to
 Section 5 for the examples and the figures. The above theorem is proved in Section 2 (Theorems \ref {th2.2},  \ref {th2.5} and Corollary \ref {th2.6}).

\medskip

The three classes of fractal squares can be determined in finite
steps. Indeed,  we show that if $F$ contains a line segment, then it
can be detected in $F_1$ (Theorem \ref{th3.3}). Also to show whether
$H^c$ has unbounded components, we make use of certain class of
paths in $H^c_k$, which can be constructed inductively and is easy
to check (Theorems \ref {th4.4}, \ref{th4.6}).  These are proved in
Sections 3 and 4. The implementation of these criteria and some
examples are given in Section 5.

\medskip

We remark that it is not straightforward to generalize the present
classification to  higher dimensions, as the technique here depends
very much on the two dimensional topology. On the other hand it is
possible to extend this consideration to disk-like self-affine tiles
(\cite {BaWa01}, \cite {DeLa11}, \cite {LeLa07}) by replacing the
square $I$ here. Indeed for the totally disconnected case, this
approach  (and for more general self-similar sets) has been taken up
by the authors \cite{LaLu12} to study the Lipschitz equivalence
problem. It may be a useful setting to study the classification
problem introduced here.

\end{section}

\bigskip

\begin{section} {\bf Classification of $F$ by connected components}

 For $H_k =
F_k + {\Bbb Z}^2$, $H = F + {\Bbb Z}^2$,  it is clear that  $q+H = H$  for $q \in {\mathbb{Z}}^2$, and $H_{k+1} \subset H_k$. Moreover we have
 \begin{equation}\label{eq2.1}
 H_{k+1} \subset  H_k/n, \quad
H\subset H/n
\end{equation}
 (as \ $ H_{k+1} = F_{k+1}+{\mathbb{Z}}^2=(F_k+{\mathcal{D}})/n+{\mathbb{Z}}^2 \subset
(F_k+{\mathbb{Z}}^2)/n=H_k/n $). For the complement, we have  $H^c =
\bigcup_{k \geq 1} H^c_k$,\  $H_k^c \subset H_{k+1}^c$\  and
\begin{equation} \label{eq2.2}
  H^c_{k}/n \subset H^c_{k+1}, \quad
 H^c/n\subset H^c.
\end{equation}

\medskip

\begin{Lem}\label{th2.1}
If there is a  component in $H^c$ that is bounded, then every
component of $H^c$ is bounded.
\end{Lem}

\begin{proof}
Let $U$ be a bounded component of $H^c$ as in the assumption. Then
$U$ is an open set since $H^c$ is open. Let $I$ denote the unit
square and let  $k$ be an integer so that $a+I \subset n^k U$ for
some $a \in {\mathbb{Z}}^2$. For an arbitrary component $V$ of
$H^c$, choose a point $b \in {\mathbb{Z}}^2$ such that $V \cap (b +
I)\ne \emptyset$. Then $(V-(b-a )) \cap (a + I)\ne \emptyset$, so
that
$$
\frac{1}{n^k}(V-(b -a )) \cap U\ne \emptyset.
$$
Since $U$ is a component and $\dfrac{1}{n^k}(V-(b -a))$ is
connected, we conclude that
$$
\frac{1}{n^k}(V-(b -a))\subset U.
$$
Therefore $V\subset b - a + n^k U$ and it is bounded.
\end{proof}

\medskip

\begin{theorem} \label{th2.2}
If $H^c$ contains a bounded component, then the diameter of every
component is uniformly bounded, say  by $\sqrt{2}(n^2+1)^2/n$.
\end{theorem}

\begin{proof}
We will prove the following claim: { \it If there is a curve $\gamma
: [0,1] \to H^c$ with}
\begin {equation} \label{eq2.3}
\hbox {diam} (\gamma)> \sqrt{2}(n^2+1)^2/n,
\end{equation}
{\it then there is a curve $\gamma'\subset H^c$ such
that $\gamma'(1)-\gamma'(0)= q\in {\Bbb Z}^2\setminus \{0\}$}.
 For such $\gamma'$, it has the property that $\gamma' + {\Bbb
L}\subset H^c$ is a continuous curve with translational period $q$
(here ${\Bbb L} := \{ mq: m \in {\Bbb Z}\}$). Note that $\gamma^* :=
\gamma' + {\Bbb L}$ behaves asymptotically like a straight line
through $\gamma'(0)$ with slope $q$. Therefore the component of
$H^c$ containing $\gamma^*$ is unbounded, and by Lemma \ref{th2.1},
every component of $H^c$ is unbounded. The theorem is a
contrapositive statement of this.

\medskip

For the $\gamma \subset H^c$, we consider the intersection of
$\gamma$ and $H_1^c$.  Let $a_1+ I/n, \dots, a_m+ I/n$ be  the
squares intersecting $\gamma$ and satisfying $a_i\in {\mathcal
D}^c/n+ {\Bbb Z}^2$.

\medskip

Case 1.  If $m=0$, i.e., there are no such squares, then for any
$b\in {\mathcal D}^c/n+ {\Bbb Z}^2$, $\gamma \cap (b + F/n) \subset
\gamma \cap (b + I/n) =  \emptyset$. Also for any $b\in {\mathcal
D}/n+ {\Bbb Z}^2$, $\gamma \cap (b+F/n )\subset \gamma \cap H
=\emptyset$. It follows that
 \begin{equation}\label{eq2.4}
(\gamma + b)\subset H^c/n  \subset H^c, \qquad \forall \ b \in {\mathbb Z}^2/n.
\end{equation}
As $\hbox {diam} (\gamma) > n^2 \cdot \sqrt 2/n$ (by (\ref{eq2.3})),
$\gamma$ intersects $n^2+1$ subsquares of size $1/n$. Hence the
pigeon hole principle implies that there exist  $b_1,\ b_2\in
{\mathbb Z}^2/n$ such that $b_2- b_1=q\in {\Bbb Z}^2\setminus\{0\}$.
Pick any $v\in {\mathcal D}^c/n$, then $\gamma-b_1 +v$ is a curve in
$H^c$ joining $v+I/n$ and $v+q+ I/n$. We let $\gamma' \subset H^c$
be the sub-curve with a slight adjustment to start at $b_1 +v$ and
to end at $b_2 +v$ (this is possible by the openness of $H^c$) and
re-parameterize it to be on $[0,1]$. Then $\gamma'(1)- \gamma'(0)=
q$ and the claim follows.

\medskip

Case 2. $0<m\leq n^2$. We can assume without loss of generality that
\begin {equation} \label{eq2.5} \|\gamma (1) - \gamma (0)\| = \hbox
{diam} (\gamma) \ ( > \sqrt{2}(n^2+1)^2/n))
\end{equation}
(otherwise we can pick $c, c'\in \gamma [0,1]$ that attain the
diameter and restrict $\gamma$ to start and end at $c,c'$).
 Clearly, there exist sub-curves $\gamma_0,\dots, \gamma_k\subset \gamma$ and $\{b_1,\dots,b_k\}\subset \{a_1,\dots,a_m\}$ such that

  (i) $\gamma_0$ joins $\gamma(0)$ and $b_1+I/n$; $\gamma_j$ joins $b_j+I/n$ and $b_{j+1}+I/n$ for $1\leq j\leq k-1$; $\gamma_k$ joins
  $b_k+I/n$ and $\gamma(1)$.

  (ii) Each $\gamma_j$ can intersect $\bigcup_{i=1}^m (a_i+I/n)$ only at its end
  points.\\
We observe that it is impossible to have $ \hbox {diam} (\gamma_{j})
\leq n^2 \cdot \sqrt{2}/n $ for all $0 \leq j\leq k (\leq m).$
Indeed in such case,  by (\ref{eq2.5}), we have
$$
\hbox {diam} (\gamma) \leq \hbox {diam} (\cup_{j=0}^k \gamma_{j})+
k\sqrt{2}/n \leq \sqrt{2}(n^2+1)^2/n.
$$
This contradicts (\ref{eq2.3}). Hence
one of the $\gamma_{j}$ satisfies  $\hbox {diam} (
\gamma_{j}) > n^2 \cdot \sqrt{2}/n$, then  the proof of Case 1 will
imply that the $\gamma'$ in the claim exists.

\medskip

Case 3. If $m > n^2$, then by the pigeon hole principle again, there exist $a_i$ and $a_j$ such that
$a_i-a_j=q\in {\Bbb Z}^2\setminus \{0\}$.
 We modify the  sub-arc of $\gamma$ to obtain a $\gamma' \subset H^c$ that
starts at the center of $a_i+I/n$ and ends at the center of
$a_j+I/n$, which satisfies the claim.
\end{proof}

\medskip

\begin{Lem}{\label{th2.3}}
If $F$ contains a line segment, then $H$ contains a straight line
with the same slope (disregarding whether the components of $H^c$ are
bounded or not).
\end{Lem}

\begin{proof}
Let $L_0$ be a line segment in $F$.  Then for  $k\geq 1$, $n^kL_0
\subset n^kH \subset H = F + {\Bbb Z}^2$. Let  $u_k$ be the
mid-point of   $n^kL_0$, and let  $v_k$ be a point in
${\mathbb{Z}}^2 $ such that $u_k-v_k\in F$. Set $L_k=n^kL_0-v _k$
and  $a_k=u_k-v_k$.  Then $\{a_k\}_k$ is a sequence in $F$. Since
$F$ is compact, there is a convergent subsequence in $\{a_k\}_k$.
For simplicity, we assume $\{a_k\}_k$ itself converges to $a\in F$.

\medskip
Let $L$ be the straight line passing through  $a$ and  parallel to
$L_0$. We assert that $L$ must lie in $H$. Indeed let $b \in L$,
then as the vector $b -a$  and $L_k$ have the same slope as $L_0$,
and each $L_k$ has length $n^k|L_0|$,  it follows that there exists
$k_0$ such that $a_k+(b-a)\subset L_k\subset H$ holds for any $k\geq
k_0$. Since $a_k+(b-a)$ converges to $a+(b-a)=b$ and $H$ is a closed
set, we get $b\in H$. Thus $L\subset H$.
\end{proof}

\begin{Cor} \label{th2.4}
If $F$ contains two non-parallel line segments, then there is a
non-trivial component of $F$ which is not a line segment.
\end{Cor}

\begin{proof}
Let $L_1, L_2$ be the two non-parallel line segments in $F$. By
Lemma \ref{th2.3}, there exist two straight lines $L_1', L_2'$ in
$H$ and they are parallel to $L_1, L_2$, respectively, and
that $L_1'\cap L_2' $   is in $F$. The corollary follows.
\end{proof}

\medskip
According to  the above results,  if $F$   possesses non-trivial
components, then either all the components  are parallel line
segments (see Figure \ref{fig.2}) or one of them is not a line
segment (see Figure \ref{fig.1}).

\medskip

\begin{theorem}\label{th2.5}
$F$ contains a non-trivial component which is not a  line segment if
and only if every component of $H^c$ is bounded.
\end{theorem}

\begin{proof}
We first show the necessity, let $ C\subset F$ be a non-trivial
component which is not a  line segment, then there are three
distinct points $a, b, c \in C$ not in a line. Suppose a component
of $H^c$ is unbounded. By the proof of Theorem \ref{th2.2},  there
exists a curve $\gamma\subset H^c$ such that $\gamma(1)-\gamma(0)=
q\in {\Bbb Z}^2\setminus \{0\}$, and  $\gamma^* := \gamma +\{ mq: m
\in {\Bbb Z}\} $ is a curve in $H^c$ behaves like a straight line
asymptotically. Hence $H$ is separated by $\gamma^*$. Assume the
line segment $[a,b]$ is not parallel to $[0,q]$ (otherwise we take
$[a,c]$ instead).  We can take a  large $k$ and a suitable $z\in
{\Bbb Z}^2$ such that $ n^k C - z \subset H$  and $n^ka -z $ and
$n^kb -z$ are separated by $\gamma^*$, which contradicts the
connectedness of $ n^k C - z$.

\medskip

For the sufficiency, suppose  $U$ is a bounded component of $H^c$.
Let $V$ be the unbound component of ${\Bbb R}^2\setminus
\overline{U}$, then $V$ is a simply connected domain. Hence the
boundary $\partial V$ is connected,  $\partial V\subset H$, and it
is not  a line segment. It follows that the non-trivial components
of $F$ can not be  parallel line segments. Hence $F$ contains a
non-trivial component which is not a line segment by Corollary
\ref{th2.4}.
\end{proof}

\medskip

\begin{Cor}\label{th2.6} If the components of $H^c$ are unbounded, then either $F$ is totally disconnected, or all non-trivial components of $F$ are parallel line segments.

In particular, in the second case, there are infinitely many
unbounded  components in $H^c$.
\end{Cor}

\begin{proof}
 Since $H^c$ contains unbounded components, by Theorem \ref {th2.5}, it follows that $F$ is either totally disconnected,
or the non-trivial components of $F$ are parallel line segments.

 To prove the last statement, let $L$ be a line in $H$,
let $u\in {\mathbb Z}^2$ be a vector such that the line segment
$[0,u]$ is not parallel to $L$. Then $L+mu$ are parallel lines for
$m\in{\mathbb Z}$.  Let $U$ be the region bounded by $L+mu$ and
$L+(m+1)u$, then $U\setminus H$ is an open set and it is not empty
since $\dim_HF<2$. Hence $U\setminus H$ contains at least one
component, and each component in $U\setminus H$ is unbounded by
Lemma \ref{th2.1}.
\end{proof}

\end{section}

\bigskip

\begin{section} {\bf $F$ containing line segments}

It is clear that  $F$ contains a vertical line segment (or
horizontal line segment) if and only if $F_1$ does. Hence we will
not include these two special cases in the following consideration.
It follows from Lemma \ref{th2.3} that $F$ contains a line segment
if and only if $H = F+ {\Bbb Z}^2$ contains a line. Suppose $L$ is a
line in $H$, then $\widetilde L= L/{\Bbb Z}^2 $ can be regarded as a
helix in the torus ${\mathbb T}^2={\mathbb R}^2/{\mathbb Z}^2$.
Since the closure of $\widetilde L$ is contained in $(F + {\Bbb
Z}^2)/{\Bbb Z}^2 (= F)$, which is a proper subset of ${\mathbb
T}^2$. The helix $\widetilde L$ is not dense in ${\mathbb T}^2$, the
slope of $L$ must be a rational number.  The same conclusion holds
for a line in $H_k$. Let us denote the slope of $L$ by $\tau=r/s$,
where $r$ and $s$ are co-prime integers and $s\geq 1$. Let $\pi:
{\mathbb R}^2\to {\mathbb R}$ be the projection along the line $L$,
that is, $\pi(x,y)=x-\tau y$, and let
\begin{equation*}
\Omega=\{\omega\in {\mathbb R}: ~ L_\omega \subset  H\},\qquad
\Omega_k =\{\omega\in {\mathbb R}: ~ L_\omega \subset  H_k\},
\end{equation*}
where $L_\omega$  denotes the line with slope $r/s$ and the
$x$-intercept $\omega$.

\medskip

\begin{Lem}
With the above notation, then

{\rm (i)} \ $\Omega+{1}/{s}=\Omega$,  \
$\Omega_k+{1}/{s}=\Omega_k$, \ and

{\rm (ii)} \ $\pi (H^c) = {\Bbb R} \setminus \Omega, \ \pi(H_k^c)=
{\Bbb R} \setminus \Omega_k$.
\end{Lem}

\begin{proof} (i) Note that $H + (1,0)= H$ and $H +(0,1) =H$, the projection $\pi$ yields
  $\Omega+1=\Omega$ and $\Omega-r/s=\Omega$ respectively. Since $r,\ s$ are co-prime, there exist integers  $k_1,k_2$ such that $k_1r+k_2s=1$. Hence
$$
  \Omega=\Omega+(k_2+k_1r/s )=\Omega+1/s.
$$
The same proof holds for $\Omega_k$. Part (ii) is clear from the definition of the projection $\pi$.
  \end{proof}

  \bigskip

Let $Tx=nx \pmod 1$ \ be a transformation on $[0,1)$, and let \ $
\widetilde \Omega=\Omega\cap [0,1),\ \ \widetilde
\Omega_k=\Omega_k\cap [0,1). $ The following lemma is crucial.

\bigskip

\begin{Lem}\label{th3.2} $\alpha\in \widetilde {\Omega}$ if and only if the orbit $\{ T^k \alpha : \ k\geq 0 \}\subset \widetilde \Omega_1.$
  \end{Lem}

  \begin{proof} Suppose $\alpha\in \widetilde {\Omega}$, then  $L_\alpha \subset H \subset H/n$ (by (\ref{eq2.1})). It follows that $ nL_\alpha\subset H$,  which implies  $T\alpha = n\alpha \pmod 1\in \widetilde {\Omega} \subset \widetilde {\Omega}_1$ and the necessity follows.

 \medskip

  For the sufficiency, we claim that  if $\ n\beta \in \Omega_k$, then  $\beta\in \Omega_{k+1}$.   Indeed we let $A$ be the set of lattice points $d\in {\mathbb Z}^2/n$ such that
  $ (d+I/n)\subset H_1$ and  $L_\beta\cap (d+I^\circ/n) \neq \emptyset$.
   That $nL_\beta= L_{n\beta}\subset H_k$ (since $n\beta\in \Omega_k$) implies that
   $$
    nL_\beta\cap \left (nd+I^\circ\right )\subset  \left (nd+F_{k}\right ), \quad d \in A.
    $$
   Taking the union of both sides for all $d\in A$ and  the closure, we obtain
  $$
  nL_\beta\subset   {\bigcup}_{d\in A}(nd+F_{k}).
  $$
  Therefore
  $
  L_\beta\subset   \bigcup_{d\in A}(d+F_{k}/n))\subset H_{k+1},
  $
  and the claim is proved.

\medskip

 For $k\geq 0$, from $T^k\alpha \in \widetilde
\Omega_1$, we infer that $n T^{k-1}\alpha = T^k\alpha + m \in
\widetilde {\Omega}_1 + m \subset \Omega_1$, where $m$ is an
integer. It follows from the claim above that  $T^{k-1}\alpha \in
\Omega_2$,  and indeed $T^{k-1}\alpha \in \widetilde \Omega_2$. By repeating
this argument, we obtain that $\alpha \in \widetilde \Omega_{k+1}$.
Hence $\alpha \in \widetilde \Omega$.

 \end{proof}

  \medskip

  The following theorem provides a simple way to determine whether the fractal square $F$ contains a line segment.

  \medskip

  \begin{theorem} \label{th3.3} $H$ contains a line  if and only if \ $\widetilde \Omega_1$  contains  either  an interval or the $T$-orbit of one point in ${\Bbb Z}/ns$ (degenerate interval).
  \end{theorem}

  \begin{proof} As $H^c_1 = (I \setminus F_1) + {\Bbb Z}^2$,  it is easy to see that $\pi(H^c_1)$ is a union of open intervals of length $(1+|r|/s)/n$, and with end points in ${\Bbb Z}/ns$. Hence $\widetilde \Omega_1$ (if nonempty) contains closed intervals with end points in ${\Bbb Z}/ns$,  or $\widetilde \Omega_1\subset {\Bbb Z}/ns$. In the later case, $\widetilde \Omega_1$ contains a $T$-orbit by Lemma \ref{th3.2}.  The necessity is proved.

  \medskip

  To prove the sufficiency,  we can identify the interval $[0,1)$ with ${\Bbb T}$ for the convenience to use the map $T (x) = nx \pmod 1$. For the degenerate case, the theorem follows immediately by Lemma \ref{th3.2}.  For the non-degenerate case, by assumption,
  we let $J_0=[\frac{m}{ns}, \frac{m+1}{ns}]$ be an interval in $\widetilde \Omega_1$. Then $T(J_0)=[\frac{m}{s}, \frac{m+1}{s}]$ is an interval in $\mathbb T$ with length $1/s$. Since ${\widetilde \Omega}_1+1/s={\widetilde \Omega}_1$,  $T(J_0)\cap \widetilde \Omega_1$ contains a translation of $J_0$,  which we denote by $J_1$. By the same argument,  there is $J_2$, a translation of $J_1$,  contained in $T(J_1)\cap \widetilde \Omega_1$. Therefore, we can find intervals $J_1, J_2,\dots$ such that they are translations  of $J_0$,   all of them are subsets of $\widetilde \Omega_1$, and $J_j\subset T(J_{j-1})$ for $ j \geq 1$.  Since there are only $n$ different translations of $J_0$, we conclude that the sequence must be eventually periodic and hence $J_{k_0}=J_{k_0+p}$ holds for some $k_0\geq 0$, $p\geq 1$. Hence $J_{k_0}\subset T^p(J_{k_0})$, and it follows that there is a $p$-periodic point $\alpha$ of $T$  in $J_{k_0}$  and $T^k\alpha \in J_{k_0+k}$ (see Sarkovskii's Theorem in \cite{Dev}). Moreover, the orbit of $\alpha$ is in ${\widetilde \Omega}_1$.
   By Lemma \ref{th3.2}, $\alpha\in \widetilde \Omega$ and  $H$ contains a line.
  \end{proof}

 \medskip

We remark that if  a line with slope $\tau$ is contained in $H$,
then $\tau=r/s$ with $1 \leq |r|+ s\leq n$, $r,s\in {\mathbb Z},
s\geq 1$.
 For otherwise, since $\pi(H_1^c)$ is a union of open
intervals of length $(1+|r|/s)/n$ and  $\pi(H_1^c)+1/s=\pi(H_1^c)$,
$(1+|r|/s)/n>1/s$ implies that $\pi(H_1^c)={\mathbb R}$ and thus
$\Omega_1 = \emptyset$.  Hence there are at most $n^2$ choices of
$\tau$. That the components of $F$ are line segments can be checked
directly on $\Omega_1$.

\end{section}

\bigskip

\begin {section} {\bf  $H^c$ and its components}

 In this section we will  study in more detail on the set $H_k^c$, and provide a criterion to determine  the boundedness of the components of $H^c$.   For $q \in {\Bbb Z}^2$, we define
$$
  H_k(0, q)=F_k+\left ( {\mathbb Z}^2\setminus \{0, q\}\right ),
$$
and denote by $H^c_k(0,q)$ its complement. Clearly $H^c_k(0,q)$ is
an open set and contains $I^\circ$ and $q+I^\circ$, and $H^c_k
\subset H^c_k(0,q)$.

\medskip

\begin{Def}\label{def4.1} A vector  $q \in {\mathbb Z}^2$  is said to be  an {\rm admissible vector of order $k \geq 1$} if $ H_k^c(0, q) $
has a component containing the open squares $I^\circ$ and
$q+I^\circ$.  We denote by $Q_k$  the set of admissible vectors of
order $k$.  By convention we let  $Q_0 = \{0, \pm e_1, \pm e_2\}$.
(Here $I= [0,1]^2$, $e_1 = (1,0), e_2=(0,1)$.)
\end{Def}

\medskip

  \noindent {\bf Remarks.} (1) It follows that $q\in Q_k$ if and only if there exists a curve $\gamma\subset H^c_k(0,q)$  that starts from
   $I^\circ$ and ends in $q + I^\circ$.  Roughly speaking,  the attachment of these two auxiliary unit squares to a curve in $H_k^c$ is
   for the sake of normalization and for convenience. For the curve $\gamma$, we can choose  one that passes through a chain of
   non-repeated squares of size $1/n^\ell$ (or $1/n^\ell$-squares) where  $0<\ell \leq k$ in  $H^c_k$, and $\gamma$ is composed of line segments connecting the centers of these squares.

 \medskip

 (2) Clearly if  $H^c_k$ has an unbounded component, then $Q_k$ is an infinite set.

\medskip

\begin{Lem} \label{th4.2} With the above notation, then

{\rm (i)} $Q_k \subset Q_{k+1}$

{\rm (ii)} The components of $H^c$ are bounded if and only if $\{Q_k\}_k$ is uniformly bounded. In this case there exists $k_0$ such that $Q_{k+1} = Q_k$ for all $k\geq k_0$.
 \end{Lem}

\begin{proof}
Part (i) follows from the fact $F_{k+1}\subset F_k$.  For part (ii),
if $\{Q_k\}_k$ is unbounded, then there exists $q_k\in Q_k$ such
that $\|q_k\|\to \infty$ as $k\to \infty$, it follows that the
corresponding sub-curve $\gamma^* \subset H_k^c$ is unbounded, and
hence the components of $H^c$ are unbounded. Also the above
implications  are reversible.
\end{proof}

\medskip

In the following, we  give a detail consideration on the structure
of $Q_k$.  Recall that  ${\mathcal D}$ is the digit set of $F$, and
let  $ {\mathcal D}^c=\{0,1,\dots, n-1\}^2\setminus {\mathcal D}$.
We define a set of vertices  by
\begin{equation*}
{\mathcal V}={\mathcal D}^c/n:= \{v_1,\dots,v_{\ell}\}.
\end{equation*}

 Let $Q$ be a subset of ${\mathbb Z}^2$  and assume that  $0 \in Q$,
we define a graph ${\mathcal G}_Q$ as follows: let $b\in {\mathbb
Z}^2$, and $u,v\in {\mathcal V}$, by an {\it edge $b$ from $u$ to
$v$}, we mean
\begin{equation}\label{eq4.1}
n(v+b-u)\in Q
\end{equation}
and denote this edge  by $(u, v; b)$.  If  $Q$ is symmetric ( i.e.,
$Q= -Q$)  and if there is an edge $(u, v; b)$, then there is an edge
$(v, u; -b)$.  By a {\it path} of ${\mathcal G}_Q$, we mean a finite
sequence $\{(u_i, u_{i+1}; b_i)\}_{i=1}^{m}\subset {\mathcal G}_Q$;
in addition, if $\sum_{i=1}^m b_i=0$, we call it a {\it $0$-path}.
This is useful for sorting the vertices into equivalence classes
(see Section 5).  A path is  a {\it loop} if $u_1 =u_{m+1} (=u)$.
In the case that $\sum_{i=1}^m b_i\not = 0$, we refer this as a {\it
non-zero} loop; otherwise we call it $0$-loop.   The edge $(u,u;0)
\in {\mathcal G}_Q$ for any $u\in  {\mathcal V}$, and we sometimes
call it a trivial loop.  Note that  $(u,u; b) \in {\mathcal
G}_{Q}$ with $b \not = 0$ is a non-zero loop.

\medskip

 We remark that  a $0$-path is not necessary a loop. The reader can refer to  Figure  \ref{fig.2}(a) in Example \ref{example5.3} for an illustration.   In the example,   $Q =Q_0 =  \{0, \pm e_1, \pm e_2\}$, then
$$
(v_1, v_1; 0), \ \ (v_1, v_2; 0), \ \ (v_2, v_3; 0) \in {\mathcal
G}_Q.
$$
They are edges associated with $0$ and are $0$-paths, the first one
is a trivial loop, but the last two are not loops.

\medskip
 We are interested in the graphs ${\mathcal G}_{Q_k}, k\geq 1$.  We remind the reader that in the sequel, a ``path" is reserved for a
sequence of edges in the graph ${\mathcal G}_{Q_k}$, and a ``curve"
is referred to a path in $H^c_k\subset {\Bbb R}^2$.  The main
motivation of this notion of graph is due to the following simple
proposition.

\medskip

\begin{Lem} \label{th4.3} Let $u, v \in \mathcal V$, $b\in {\mathbb Z}^2$.

{\rm (i)}  If $(u,v; b) \in {\mathcal G}_{Q_{k-1}}$,  then there is
a curve $\gamma\subset H_k^c$  joining $u +  {I^\circ}/n$ and \ $b +
v + {I^\circ}/n$.

{\rm (ii)}  Conversely, if there is a curve $\gamma\subset H_k^c$   joining
 $u +  {I^\circ}/n$ and \ $b + v + {I^\circ}/n$, and  the curve does not intersect the closure of any other  $1/n$-squares in  $H^c_k$, then   $(u,v; b) \in {\mathcal G}_{Q_{k-1}}$.
\end{Lem}

\begin{proof} (i) Note that $(u,v; b) \in {\mathcal G}_{Q_{k-1}}$ means $q=n(v+b-u) \in Q_{k-1}$, which means there is a curve $\gamma'$ connecting
$I^\circ$ and $q+I^\circ$ in $H^c_{k-1}(0,q)$.  Then
$\gamma'+ nu$ is a curve connecting $nu + I^\circ$ and $q + nu +
I^\circ$ in  $H_{k-1}^c(0,q) +nu$. We claim that
$$
\frac 1n H_{k-1}^c(0,q) +u \subset H_k^c.
$$
 This will imply $\gamma=\gamma'/n + u$ is a curve joining $u +
I^\circ/n$ and $v +b+I^\circ/n$ in $H^c_k$ and (i) follows.

 To prove the claim we first observe that for $u\in {\mathcal V}$,
then $nu \in {\mathcal D}^c$,  and it is easy to check that
${\mathcal D} \subset ({\Bbb Z}^2 \setminus \{0, q\} )+nu$. Hence
$$
nH_{k} =   (F_{k-1} + {\mathcal D}) + n{\Bbb Z}^2 \subset  F_{k-1} +
({\Bbb Z}^2 \setminus \{0,q\})+nu =  H_k(0,q)+nu.
$$
The claim follows by taking the complement of the above.

\medskip

(ii) Suppose $\gamma$ is a curve in $H_k^c$ as in the lemma.
 Let $\gamma^*=\gamma(t_1,t_2)$ be an open sub-arc of $\gamma$ such that
$\gamma^*$ does not intersect the two squares $u+I/n$ and $v+b+I/n$.
  By the same argument as Case 1 in Theorem \ref {th2.2}, we have that  $\gamma^*\subset H_{k-1}^c/n$ as in (\ref{eq2.4}), which implies that
 $n\gamma^*-nu\subset H_{k-1}^c$. It is seen that  $n\gamma^*-nu$  can be extended to  $I^\circ$ and $n(v+b-u)+I^\circ$,
 hence $n(v+b-u)\in Q_{k-1}$, which implies $(u,v;b)\in {\mathcal G}_{Q_{k-1}}$.
\end{proof}

\medskip

 By using (\ref{eq4.1}), we introduce several auxiliary classes of edge
sets. Let $\widetilde {\mathcal V}=\big\{h/n:
h\in\{0,1,\dots,n-1\}^2\big\}$ and define
$$
\widetilde {\mathcal G}_Q = \{(u, v; b): u, \ v
\in \widetilde {\mathcal V}\}.
$$
Then for an edge $(u, v; b) \in \widetilde {\mathcal G}_{Q_{k-1}}$, it has the
same property as in Lemma \ref{th4.3} except by replacing $H^c_k$
with $H_k^c\cup (u+ I^\circ/n)\cup (v+b+I^\circ/n)$  (notice that
$u+I^\circ/n$ and $(v+b+I^\circ/n)$ are subsets of $H_k^c$ when
$u,v\in {\mathcal V}$).

Especially, we set
$$
{\mathcal G}'_Q = \{(u, v; b): u \in \widetilde {\mathcal V}, \ v
\in {\mathcal V}\} \quad \hbox {and} \quad  {\mathcal G}_Q^{''} =
\{(u, v; b): u \in {\mathcal V}, \ v \in \widetilde {\mathcal V}\}.
$$
Now we can give the inductive relationship of
$Q_k$.

\medskip

\begin {theorem} \label{th4.4} For  any $k \geq 1$, $
Q_k$ equals the set  of  $q=  b_0 + \cdots +b_{m}$ from the path
$\{(u_i, u_{i+1}; b_i)\}_{i=0}^{m}$ with $m=0,1$ or
\begin{equation} \label{eq4.2}
(u_0, u_1; b_0) \in {\mathcal G}'_{Q_{k-1}}, \ \{(u_i, u_{i+1};
b_i)\}_{i=1}^{m-1} \subset {\mathcal G}_{Q_{k-1}},  \ (u_{m},
u_{m+1}; b_{m}) \in {\mathcal G}^{''}_{Q_{k-1}}.
\end{equation}
\end{theorem}

\begin{proof}
 If $m=0$, then $(u_0,u_1;b)$ is an edge in  $\widetilde {\mathcal G}_{Q_{k-1}}$.  By the remark above,
$u_0+I^\circ/n$ and $u_1+b+I^\circ/n$ are in the same component of
$H_k^c\cup (u_0+ I^\circ/n)\cup (u_1+b+I^\circ/n)\subset H_k^c(0,b)$
and hence $b\in Q_k$. Similarly for $m=1$.

\medskip

Assume $m\geq 2$ and $\{(u_i, u_{i+1}; b_i)\}_{i=1}^{m-1} \subset
{\mathcal G}_{Q_{k-1}}$, we have by Lemma \ref{th4.3} that
 there is a component in $H^c_k$ containing
$$
u_1 + I^\circ/n, \ \ b_1+u_2 + I^\circ/n,\ \dots ,\ \ (b_1+ \cdots
+b_{m-1}) + u_{m}+ I^\circ/n.
$$
So there is a curve $\gamma$ joining $u_1+I^\circ/n$ and $(b_1+
\cdots +b_{m-1}) + u_{m}+ I^\circ/n$.  We add an initial curve
$\gamma'$  and a final curve  $\gamma''$ corresponding to $(u_0,
u_1; b_0)$ and $(u_{m}, u_{m+1}; b_{m})$ respectively,  and
$\gamma'\cup \gamma\cup \gamma''$ is a new curve joining the
following squares (replacing $u_0+I^\circ/n$ and $u_{m+1}+I^\circ/n$
at the two ends by $I^\circ$, since they are subsets of $I^\circ$):
$$
 I^\circ, \  b_0+u_1 + I^\circ/n,\ \dots , \ (b_0+\cdots +b_{m-1})+ u_m + I^\circ/n, \ (b_0+ \cdots +b_{m}) + I^\circ.
$$
It is in $H^c_k(0,q)$, and this  implies that $ q =b_0 + \cdots +
b_{m}$ is in $Q_k$.

\medskip

Conversely, let $q \in Q_k$, then there exists a simple curve
$\gamma \subset H^c_k(0,q)$ connecting $I^\circ$ and $q+I^\circ$ (as
in Remark (1) of Definition \ref{def4.1}). Let $\gamma^*$ be the
part of the curve by deleting the parts in $I^\circ$ and
$q+I^\circ$. Let $\{a_i+I/n\}_{i=1}^{m}$ be the $1/n$-squares that
intersect $\gamma^*$  with $a_i\in D^c/n+{\mathbb Z}^2$ (if exist;
otherwise, reduce to $m=0$). Without loss of generality, we may
assume that $\gamma$ passes each square $a_i+I^\circ/n$ only once
and that $\{a_i+I^\circ/n\}_{i=1}^{m}$ are arranged in the order
according to the advance of $\gamma$. Then we can write them as
$$
a_i= c_i + u_i,  \qquad u_i \in {\mathcal V}, \ \ c_i
\in {\Bbb Z}^2 .
$$

We add in two more $1/n$-squares $a_0+I^\circ/n$ and
$a_{m+1}+I^\circ/n$ as follows: since the curve $\gamma^*$ has an
extension into $I^\circ$ and is contained in $H^c_k(0,q)$, we use
$a_0+I^\circ/n=u_0+I^\circ/n$ to denote the $1/n$-square in
$I^\circ$ that contains the extension  where $a_0\in \widetilde
{\mathcal V}$. Similarly we choose $a_{m+1}+I^\circ/n=q
+u_{m+1}+I^\circ/n$ in $q+I^\circ$ where $ u_{m+1}\in \widetilde
{\mathcal V}$. Let
$$
b_0 =c_1, \ \  b_i = c_{i+1} -c_{i}, \ \ 1 \leq i \leq m -1  \quad
\hbox {and} \quad   b_{m} = q - c_{m}.
$$
It follows that (by Lemma \ref{th4.3} (ii)) the sequence $\{(u_i,
u_{i+1}; b_i)\}_{i=0}^{m}$ satisfies (\ref{eq4.2}), since the curve
between $a_{i}+I^\circ/n$ and $a_{i+1}+I^\circ/n$ does not intersect
any other $1/n$-squares.
\end{proof}

\medskip

\begin{Cor}
If $Q_k=Q_{k+1}$ for some $k\geq 1$, then $Q_k=Q_{k+p}$ for all
$p\geq 1$.

\end{Cor}

\medskip

We remark that  ${\mathcal V}\subset \widetilde {\mathcal V}$,  hence a path $\{(u_i,
u_{i+1}; b_i)\}_{i=0}^{m}$ in ${\mathcal G}_{Q_{k-1}}$  by itself
satisfies (\ref {eq4.2}) by treating $u_1, u_m$ as $u_0, u_{m+1}$.
For brevity, we write
 $$
{\mathcal G}_k := {\mathcal
G}_{Q_{k}}\quad\text{and}\quad \widetilde {\mathcal G}_k :=
\widetilde {\mathcal G}_{Q_{k}}.
$$
 The key role of the graph $\widetilde {\mathcal G}_k$ is to
illustrate the relation between $Q_{k}$ and  $Q_{k+1}$ as in Theorem
\ref {th4.4}.  However,  to determine  the boundedness of the
components of $H^c$, only the information of  the graph ${\mathcal
G}_k$ is needed.

\medskip

\begin{theorem} \label{th4.6}
The components of $H^c$ are unbounded if and only if there is a
non-zero loop in some ${\mathcal G}_k$.
\end{theorem}

\begin{proof}
For the sufficiency,  by Lemma \ref{th4.3}, the assumption implies
that there is a curve  $\gamma$ in $H_{k+1}^c $ satisfying
$\gamma(1)-\gamma(0)=b \in {\mathbb Z}^2\setminus \{0\}$.
This implies   $ H_{k+1}^c \subset H^c$ has an unbounded component,
and by Lemma \ref{th2.1}, all the components of $H^c$ are unbounded.

\medskip

For the necessity, if the components of $H^c$ are unbounded, then there exists a curve
$\gamma\subset H^c$ such that $\gamma(1)-\gamma(0)=q\in {\mathbb Z}^2\setminus \{0\}$.
Let $\gamma^*=\gamma+\{mq;~ m\geq 0\}.$

\medskip

Case 1. If $\gamma^*$ intersects a  square  $a+I/n$ with $a\in D^c/n+{\mathbb Z}^2$, then $\gamma^*$
 also intersects $a+q+I/n$.
 Let $\gamma'\subset \gamma^*$ be a sub-arc joining
 $a+I/n$ and $a+q+I/n$, a similar argument as the second part of the proof of Theorem \ref{th4.4} implies there is a non-zero loop in
  ${\mathcal G}_{k}$.

\medskip

Case 2.  If $\gamma^*$ does not intersect any square $a+I/n$ with
$a\in D^c/n+{\mathbb Z}^2$,
  then $\gamma^*\subset H^c/n$. Hence $\gamma^*+b\subset H^c/n\subset H^c$ for any $b\in {\mathbb Z}^2/n$ (by (\ref{eq2.4})).
  Pick any $u\in D^c/n$, we can choose $b$ so that $\gamma^*+b$ passes
 $u+I/n$ and the result follows by  Case 1.
\end{proof}

\medskip

As a direct consequence,  if there is no  non-zero loop in some
${\mathcal G}_k$ and ${\widetilde {\mathcal G}_k} ={\widetilde
{\mathcal G}_{k+1}}$, then there is no  non-zero loop in all
${\mathcal G}_k$, hence all the components of $H^c$ are bounded.
Another simple observation is,  if $b\in Q_k$ and $b\in n{\mathbb
Z}^2\setminus \{0\}$, then in  ${\mathcal G}_{k}$ we have a non-zero
loop $(v, v; b/n)$ for any $v\in {\mathcal V}$, hence the components
of $H^c$ are unbounded.  Moreover, Lemma \ref{th4.2} implies that if
${\mathcal G}_k$ has infinitely many edges (equivalently, $Q_{k+1}$
is unbounded),  then the components of $H^c$ are unbounded as well.

\medskip

These criteria provide a convenient way to classify the topology of
the fractal square $F$, which will be explained by using several
instructive examples  in the next section.

\end{section}

\bigskip

\begin{section}{\bf Algorithm and Examples}

  In Section 2, we have shown that the  fractal square
$F$ can be classified  into three types according to their
topological structure: (i) $F$ is totally disconnected; (ii) the
non-trivial components of $F$ are parallel line segments; and (iii)
$F$ contains a non-trivial component that is not a line segment. For
some of the simple cases, it is easy to inspect these types
directly. However, in general, it is difficult to see the topology
of $F$ in an obvious manner. By making use of the construction in
Section 4, it is possible to devise an algorithm to obtain the
classification. The basic idea of the algorithm is as follows

 \begin{equation} \label {eq5.1}
Q_0 \ \ \stackrel {(\ref{eq4.1})}  \longrightarrow \ {\widetilde
{\mathcal G}_0} \ \stackrel { {Theorem} \ \ref{th4.4}}
\longrightarrow \ Q_1 \ \stackrel {(\ref{eq4.1})} \longrightarrow \
 {\widetilde {\mathcal G}_1} \ \cdots
\end{equation}
Then we can use Theorem \ref {th4.6} to determine whether the
components of $H^c$ are bounded, which distinguishes type (iii) from
types (i) and (ii). By Theorem \ref{th3.3}, we can separate types
(i) and (ii).  The following proposition justifies the finiteness of the
algorithm described in   (\ref{eq5.1}).

\medskip

\begin{Prop}\label{prop5.1}
There exists $k\  (\leq 38n^{10})$ such that the process in (\ref
{eq5.1}) ends; at such $k$, either

{\rm (i)} ${\mathcal G}_{k}$ contains a non-zero loop, or

{\rm (ii)} there is no non-zero loop in ${\mathcal G}_{k}$ and
${\widetilde {\mathcal G}_{k}}= {\widetilde {\mathcal
G}_{k-1}}$.
\end{Prop}

\begin{proof} Let $k$ be the maximal integer such that there is no non-zero loop
in ${\mathcal G}_{k}$ and ${\widetilde {\mathcal G}_{k}} \ne
{\widetilde {\mathcal G}_{k-1}}$.  Then $ {\widetilde {\mathcal
G}_0}\subsetneq {\widetilde {\mathcal G}} _1 \subsetneq \cdots
\subsetneq {\widetilde {\mathcal G}_{k}}$, and $\#{\widetilde
{\mathcal G}_{k-1}}\geq k-1+n^2$ since $ {\widetilde {\mathcal
G}_0}$ contains at least $n^2$ trivial edges.

\medskip

As there is no non-zero loop in ${\mathcal G}_{k}$, the components
of $H_{k}^c$ are bounded (otherwise  there exists $b\in Q_k$ such
that $b\in n{\mathbb Z}^2\setminus \{0\}$).
For edges $(u_0,u; b)\in {\widetilde {\mathcal G}}_{k-1}$, there is
a curve $\gamma\subset   H_{k}^c$ joining $u_0+I$ and $u+b+I$. It
follows by Theorem \ref{th2.2}  that $\text{diam} (\gamma) \leq
\sqrt{2}(n^2+1)^2/n$. Hence
$$\|b\|\leq \text{diam} (\gamma) + 2\sqrt 2\leq \sqrt{2}(n^3+2n+2+1/n)\leq 29\sqrt 2 n^3/16.$$
That implies  $\#{\widetilde {\mathcal G}_{k-1}}\leq
(2\|b\|+1)^2\cdot n^4< 38 n^{10}$, and $k< 38n^{10}-n^2+1<
38n^{10}$. The proposition follows.
\end{proof}

\medskip

If ${\mathcal G}_k$ has no non-zero loops, then the number of
non-zero paths as in Theorem \ref{th4.4} is  uniformly bounded.
Therefore, to produce $Q_{k+1}$, we only need  to check finitely
many steps.  We also point  out that the estimate of the steps in
Proposition \ref{prop5.1} is very rough. In practice, the number of
steps really needed is far less, as is seen in the following
examples.

\medskip

\begin{Exa} {\rm  The fractal square in Figure \ref{fig.1} is the well-known Vicsek fractal. It is
clear that $F$  contains dendrite curves.  It is also easy to see
that $H_k$ contains horizontal and vertical lines which divide
$H_k^c$ into bounded components for any $k\geq 1$.}
\end{Exa}

\begin{figure}[h]
 \centering
\subfigure[]{
 \includegraphics[width=3.5cm]{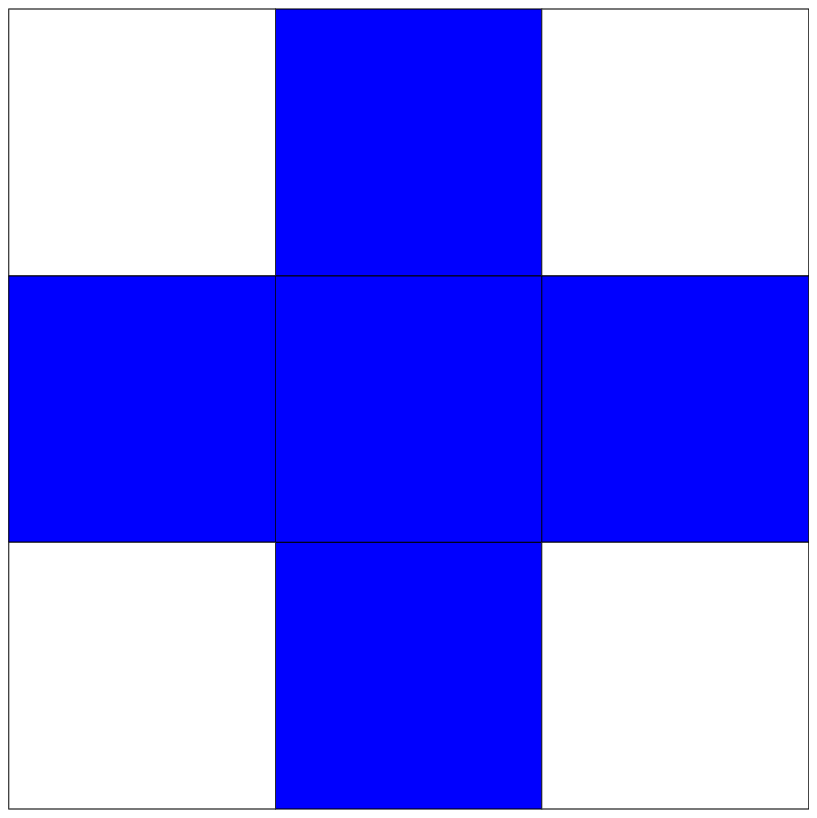}
 }
 \quad
 \subfigure[]{
 \includegraphics[width=3.5cm]{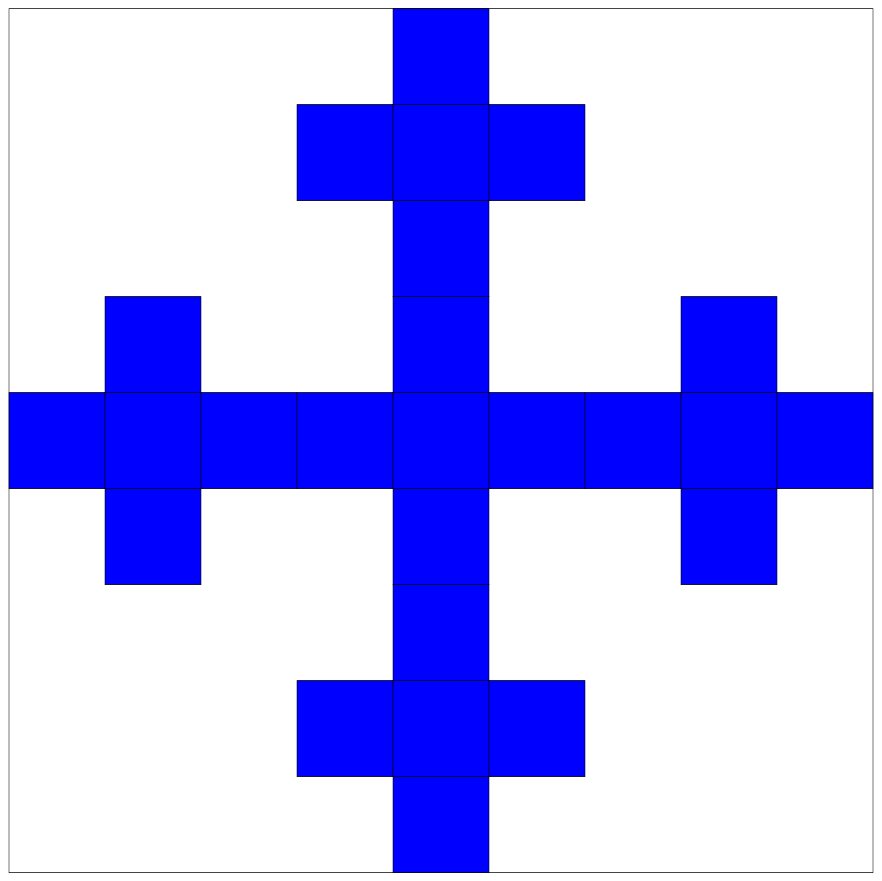}
 }
 \quad
 \subfigure[]{
 \includegraphics[width=3.5cm]{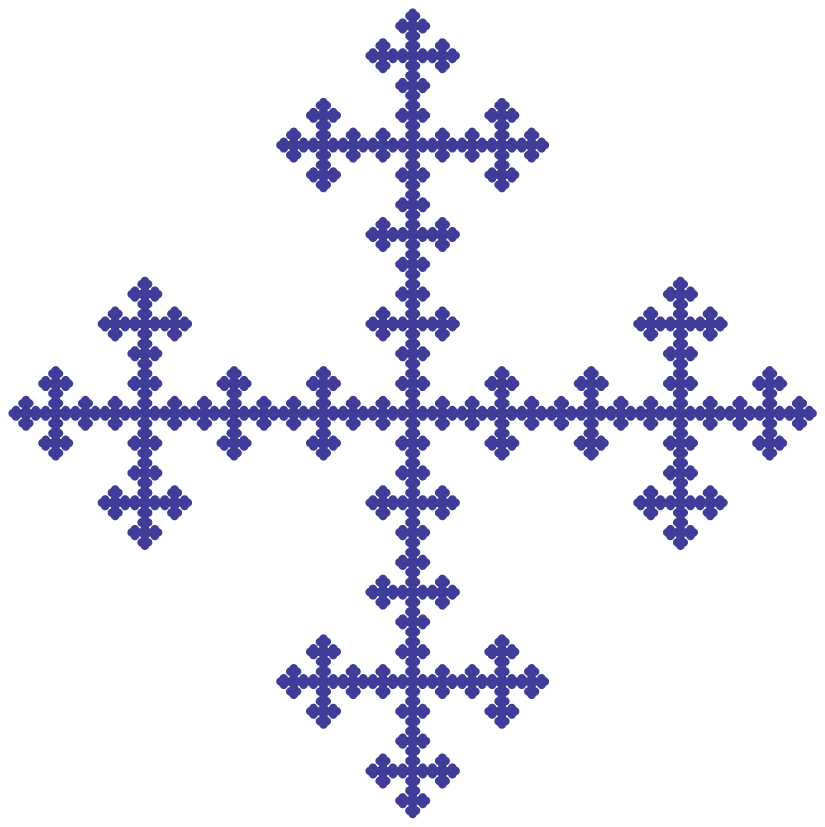}
 }
 \caption{Vicsek fractal}\label{fig.1}
\end{figure}

\medskip

\begin{Exa} \label{example5.3} {\rm
 Consider the fractal square in Figure \ref{fig.2}, the vertex set ${\mathcal V}= {\mathcal D}^c/n = \{v_1,v_2,v_3,v_4\}$  is depicted in
Figure \ref{fig.2}(a).  It is easy to see that $H_1$ contains the
line $x=y$, and $\widetilde\Omega_1 = \{0\}$, hence $H$ contains the
line by Theorem  \ref{th3.3}. Moreover, this line is also a
component of $H$. It follows from Corollary \ref {th2.6} that the
non-trivial components of  $F$ are parallel line segments.}
\end{Exa}

\begin{figure}[h]
  \centering
\subfigure[]{
  \includegraphics[width=3.5cm]{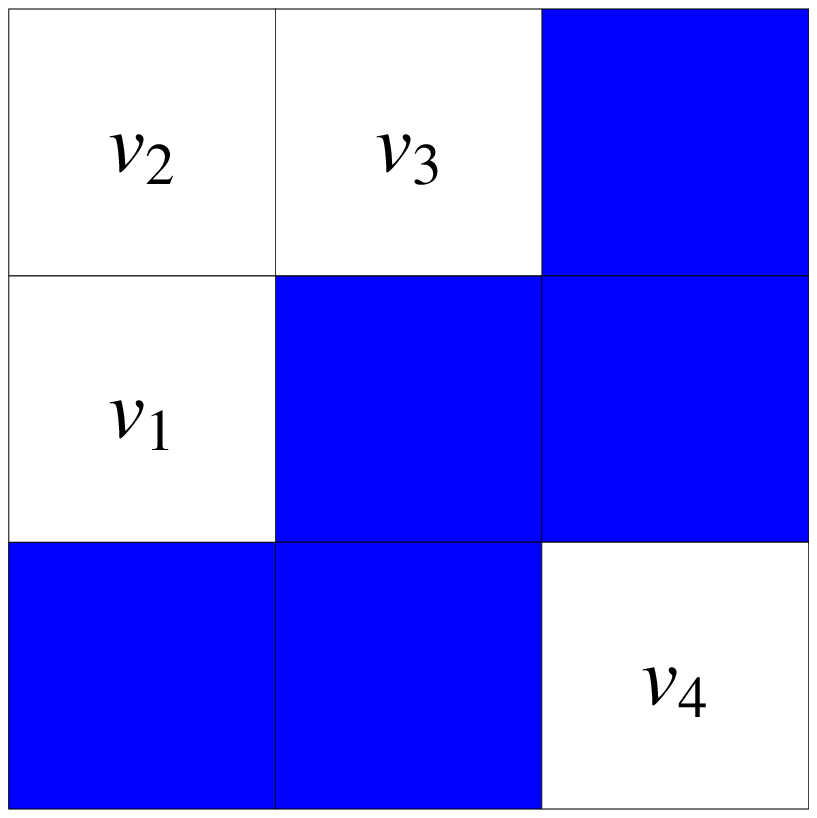}
  }
 \quad
 \subfigure[]{
  \includegraphics[width=3.5cm]{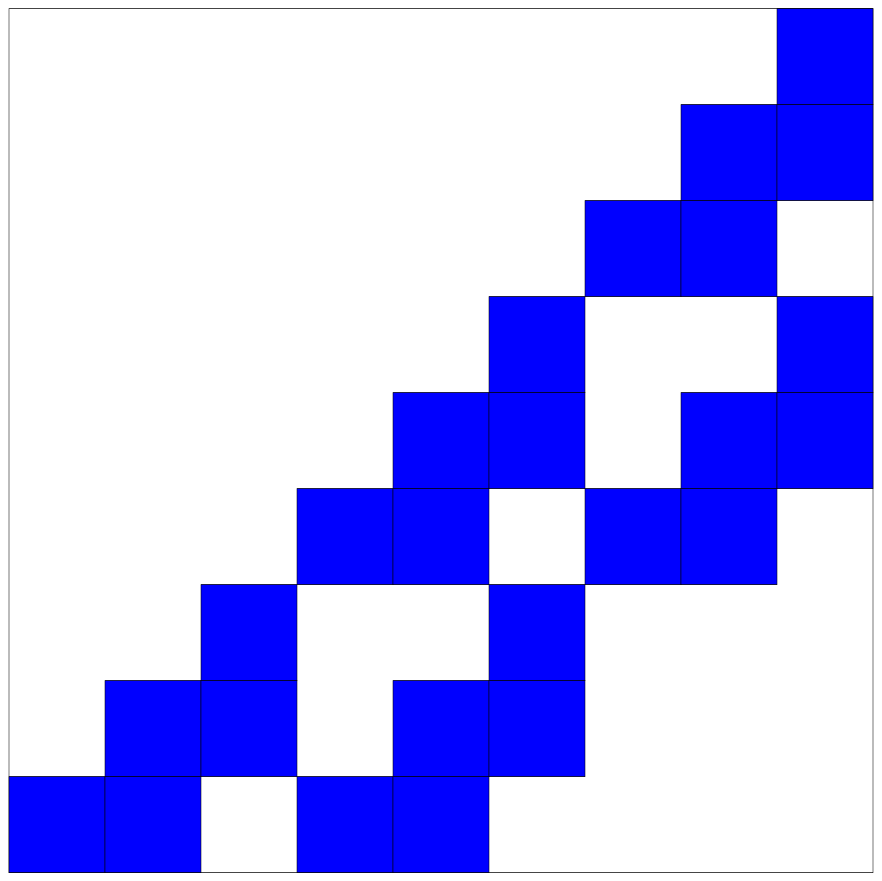}
  }
 \quad
 \subfigure[]{
  \includegraphics[width=3.5cm]{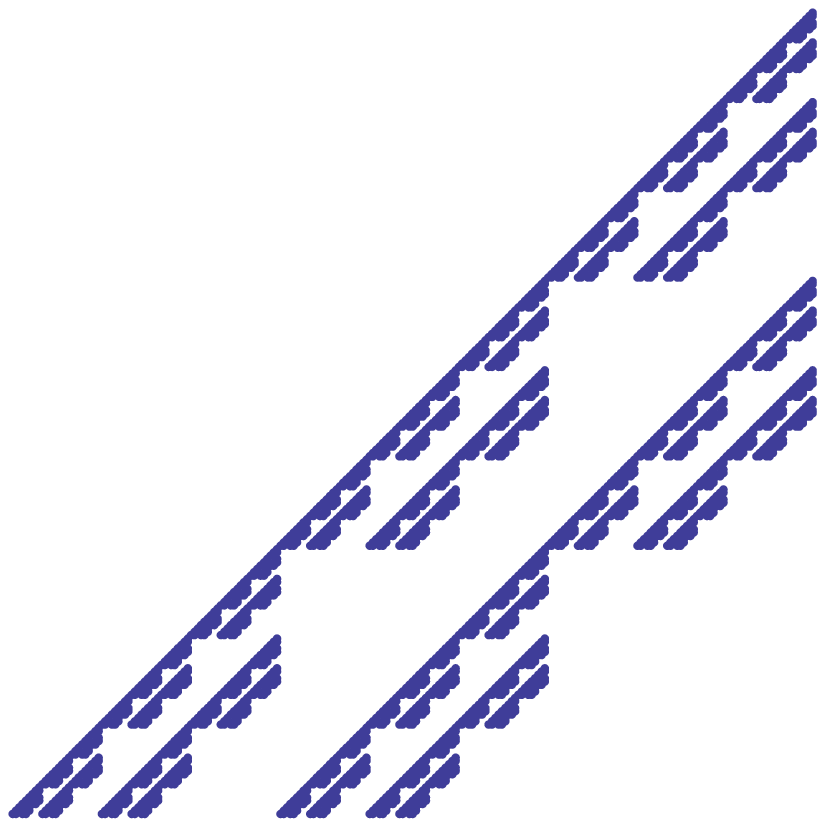}
  }
 \caption{$F$ composed of parallel line segments}\label{fig.2}
\end{figure}

\medskip

In the sequel,  we use two examples to demonstrate the inductive
method for the classification derived from (\ref {eq5.1}). Before we
do that, we simplify the graph $\widetilde {\mathcal G}_k$ by
identifying some of the vertices as follows.

\bigskip

\noindent {\bf 1. Identifying vertices in $\widetilde {\mathcal
V}$:} \quad  We introduce an abstract vertex and denote it by
$\varepsilon_0$. Set ${\mathcal V}^0=\{\varepsilon_0\}\cup{\mathcal
V}$. Then we define a graph ${\mathcal G}^0_Q$  to be an extension
of ${\mathcal G}_Q$ by adding  the following edges: for $u\in
{\mathcal V}$, $(\varepsilon_0,u; b)\in {\mathcal G}^0_Q$  if and
only if $(u_0,u;b)\in \widetilde {\mathcal G}_Q$ for some $u_0\in
\widetilde {\mathcal V}$,
  and  $(u,\varepsilon_0; b)\in {\mathcal G}^0_Q$  is defined similarly; moreover,
  $(\varepsilon_0,\varepsilon_0;b)\in {\mathcal G}^0_Q$  if and only if
 $(u_1,u_2;b)\in \widetilde {\mathcal G}_Q$ for some $u_1,u_2\in \widetilde {\mathcal V}$.  Write
${\mathcal G}^0_k=  {\mathcal G}^0_{Q_{k}}$,   and note that

\noindent (i) The restriction of  ${\mathcal G}^0_k$ to ${\mathcal V}$ is   ${\mathcal G}_k$;

\noindent (ii) $Q_{k+1}=\left\{\sum_{i=0}^m b_i:\
\{(u_i,u_{i+1};b_i)\}_{i=0}^m \ \text{is a  loop containing }\
\varepsilon_0\ \text{in}\  {\mathcal G}^0_{k}\right\}$.

\medskip

\paragraph{\textbf{2. Identifying vertices in ${\mathcal V}$:}} \quad  We start with
${\mathcal G}_0$, two vertices $u,v \in {\mathcal V}$ are said to be
\emph{equivalent} in ${\mathcal G}_0$ if there is a $0$-path joining
$u, v$ (i.e.,  there is a finite sequence $\{(u_i, u_{i+1};
b_i)\}_{i=1}^{m}\subset {\mathcal G}_0$ such that $u=u_1, v=
u_{m+1}$ and  $\sum_{i=1}^m b_i=0$);  note that in this case  $u +
I^\circ/n$ and $v+I^\circ/n$ are in $I^\circ$ and are connected in
$H^c_1$.  We use $[u]$ to denote the equivalence class containing
$u$,  and ${\mathcal V}_0^*$  the set of equivalence classes. We
introduce a graph  ${\mathcal G}_0^*$ on ${\mathcal V}_0^*$,  call
it a \emph{reduced graph} of  ${\mathcal G}_0$, by defining edges
$([u], [v]; b)\in {\mathcal G}_0^*$ if there exist $u'\in [u]$ and
$v'\in [v]$ such that $(u',v'; b)\in {\mathcal G}_0$.

Similar to Part 1, we  define a reduced graph  ${\mathcal G}_0^{0*}$
on ${\mathcal V}_0^{0*}=\{\varepsilon_0\}\cup {\mathcal V}_0^*$.
Inductively, we can perform the same reduction on each ${\mathcal
G}_k$ (resp. ${\mathcal G}_k^{0*}$) and obtain a compatible sequence
of vertex sets ${\mathcal V}_k^*$ (resp. ${\mathcal V}_k^{0*}$) and
reduced graphs ${\mathcal G}_k^*$ (resp. ${\mathcal G}_k^{0*}$).

\bigskip

\begin{Exa}{\rm
Consider the fractal square in Figure \ref{fig.totally}, the vertex
set ${\mathcal V} =\{v_1,v_2,v_3,v_4\}$ is given as in Figure
\ref{fig.totally}(a).

\bigskip

\begin{figure}[h]
  \centering
  \subfigure[]{
  \includegraphics[width=3.5cm]{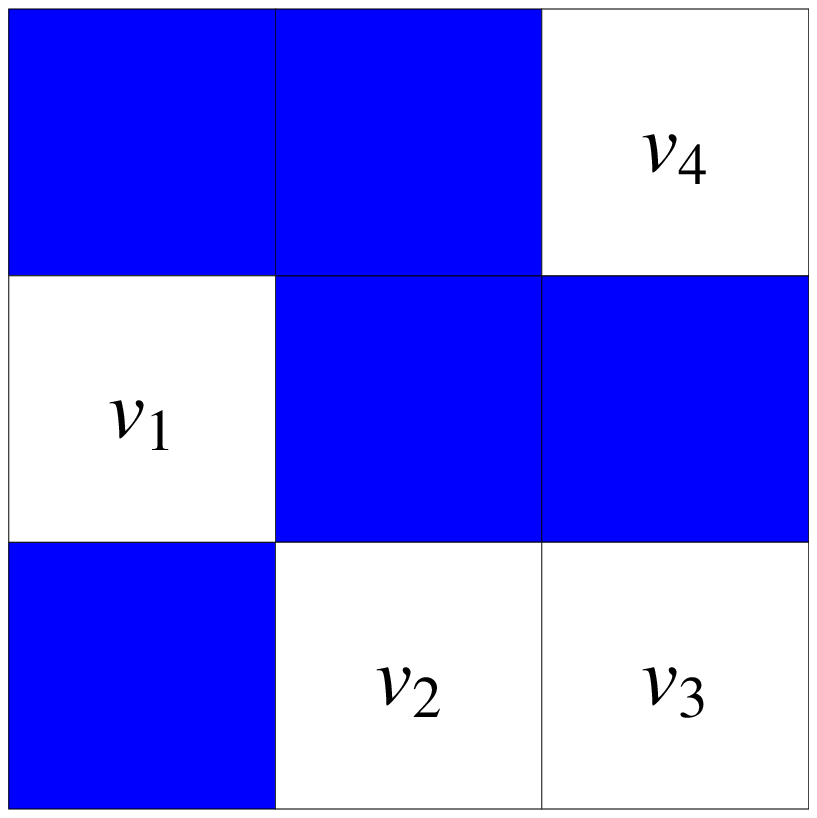}
  }
 \quad
  \subfigure[]{
  \includegraphics[width=3.5cm]{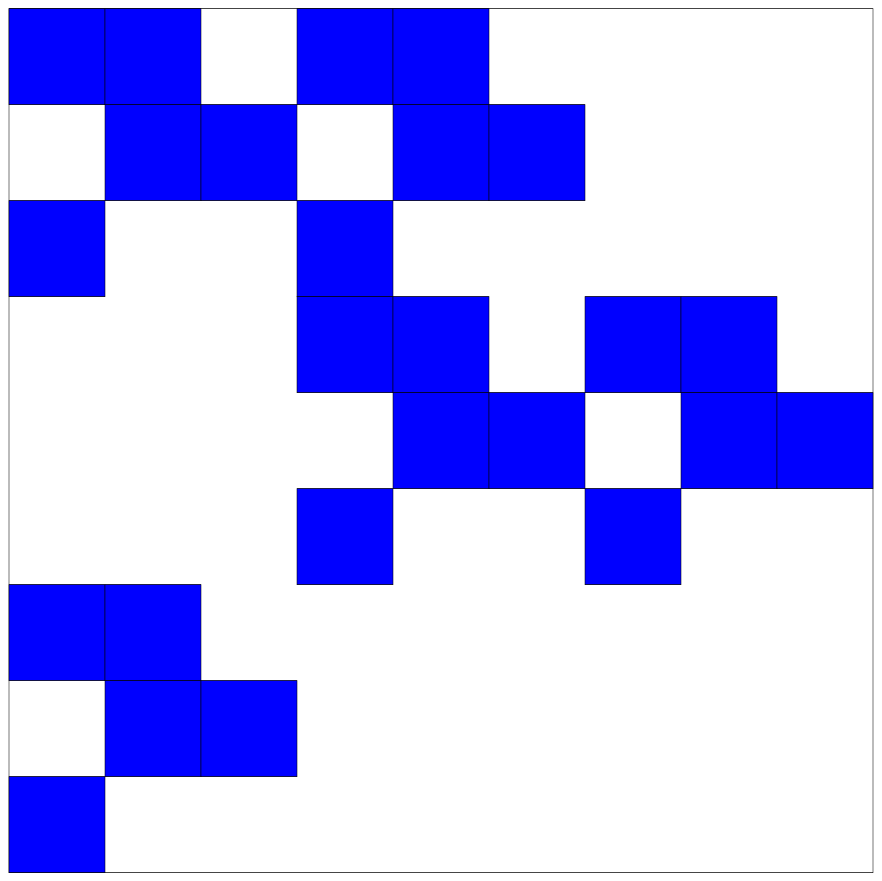}
  }
 \quad
 \subfigure[]{
 \includegraphics[width=3.5cm]{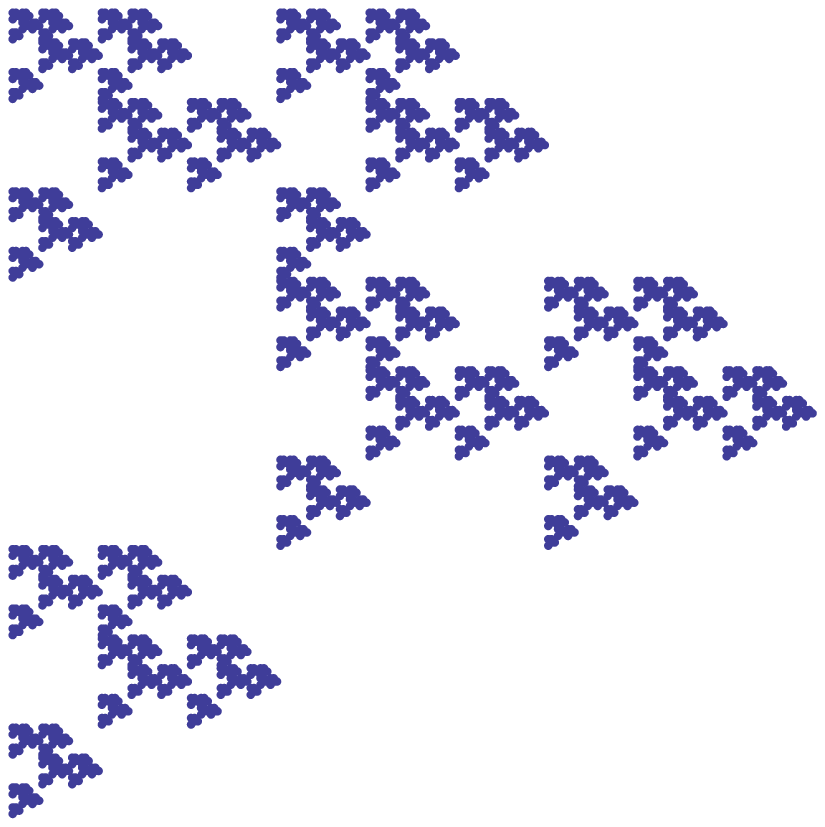}
 }
\caption{Totally disconnected $F$}\label{fig.totally}
\end{figure}

\medskip

Let $Q_0=\{0, \pm e_1, \pm e_2\}$.  Clearly $v_2,v_3$ are equivalent
in ${\mathcal G}_{0}$, we denote  the class by $[v_2]$.  Then
${\mathcal V}_0^{0*} = \{\varepsilon_0, v_1, [v_2], v_4\}$, the
non-trivial edges in the reduced graph ${\mathcal G}_0^{0*}$ are
\begin{eqnarray*}
&&  (\varepsilon_0, [v_2]; - e_1);\ \  (\varepsilon_0,[v_2]; e_2);\ \ (\varepsilon_0, v_1; e_1);\\
&&  (\varepsilon_0, v_4; -e_2);\ \ (\varepsilon_0, v_4; -e_1);\ \
([v_2], v_4; -e_2).
\end{eqnarray*}
The two non-zero paths  satisfying (\ref{eq4.2}) are
$$
\{(\varepsilon_0, [v_2]; -e_1),\ ([v_2], \varepsilon_0; -e_2)\} \ \
\hbox {and} \ \ \{(\varepsilon_0, v_4; -e_2),\ (v_4, \varepsilon_0;
e_1)\}
$$
which give $q=-e_1-e_2=-(1,1)$ and $e_1-e_2=(1,-1)$. Hence
$Q_1=Q_0\cup \{\pm (1,1),\ \pm (1,-1)\}.$

\medskip

Next, for  ${\mathcal G}_1 :={\mathcal G}_{Q_1}$, there are new
edges $ (v_1,v_2; 0),\ \  (v_1,v_4; -e_1), $  and their reverse
edges. Hence $v_1,v_2,v_3$ are equivalent in ${\mathcal G}_1$, we
denote by $[v_1]$ the equivalence class. The vertex set of
equivalence classes is ${\mathcal V}_1^{0*} = \{\varepsilon_0,
[v_1], v_4\}$, and the reduced graph ${\mathcal G}_1^{0*}$ consists
of edges
\begin{eqnarray*}
 \{(\varepsilon_0,\varepsilon_0; b):  \ b\in Q_1\};\ \  ([v_1], v_4; -e_1);\ \ ([v_1], v_4; -e_2).
\end{eqnarray*}
This yields a non-zero loop $\{([v_1], v_4; -e_1),\ (v_4, [v_1];
e_2)\}$ in ${\mathcal G}_1^{*}$. Therefore ${\mathcal G}_1$ has a
non-zero loop, and  the components of $H^c$ are unbounded by Theorem
\ref{th4.6}.

\medskip
On the other hand, it is easy to observe that $\Omega_1=\emptyset$
for any slope $\tau$, hence there are no line segments in $F$ by
Theorem \ref{th3.3}. Consequently, $F$ is totally disconnected.
\hfill $\Box$ }

\end{Exa}

\medskip

Finally, we consider one more example of which the classification is
not so obvious by observation, and it relies on using the above
technique to check the $Q_k$ and  ${\mathcal G}_k^{0*}$.

\medskip

\begin{Exa}\label{example}{\rm
Let $F$ be the fractal square in Figure \ref{fig.ex}, and the vertex
set ${\mathcal V} =\{v_1,\dots,v_9\}$ is as in Figure
\ref{fig.ex}(a). We only sketch the main steps and  omit the
straightforward but tedious verification.  The details can be found
in \cite{Lu12}.

\bigskip

\begin{figure}[h]
  \centering
  \subfigure[]{
 \includegraphics[width=3.5cm]{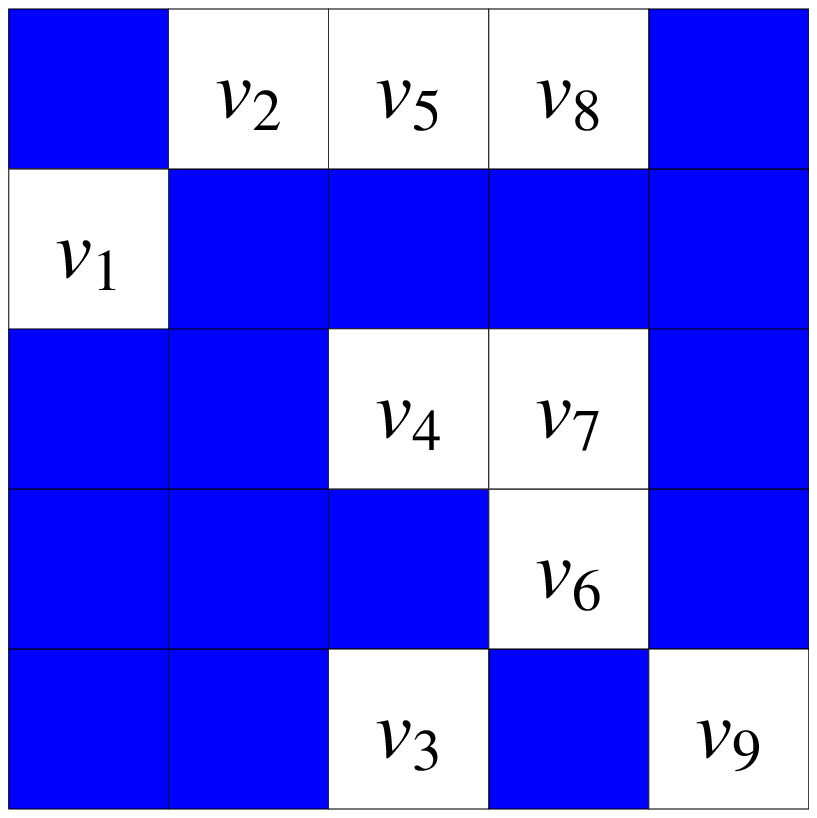}
 }
 \quad
 \subfigure[]{
 \includegraphics[width=3.5cm]{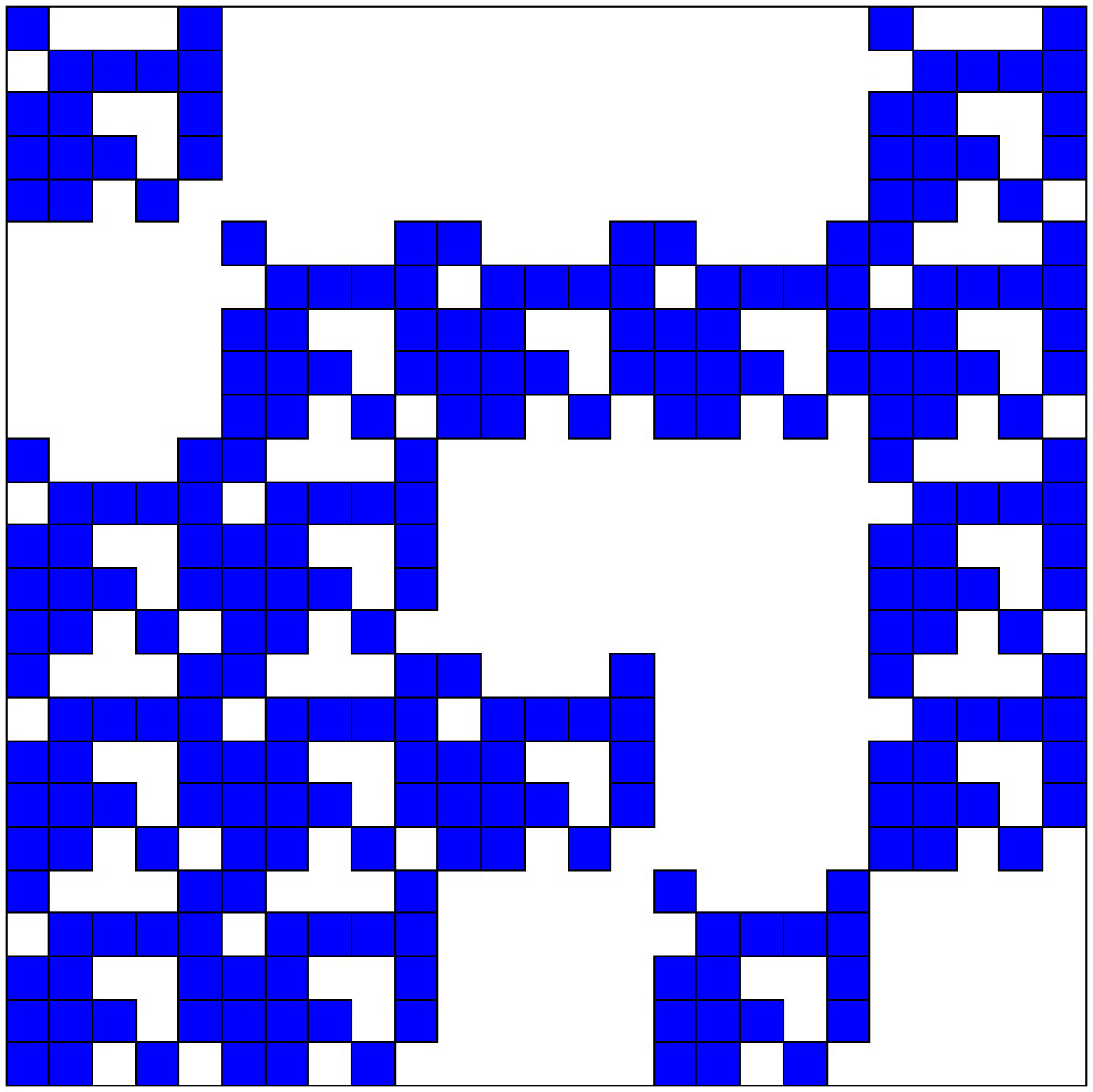}
 }
 \quad
\subfigure[]{
 \includegraphics[width=3.5cm]{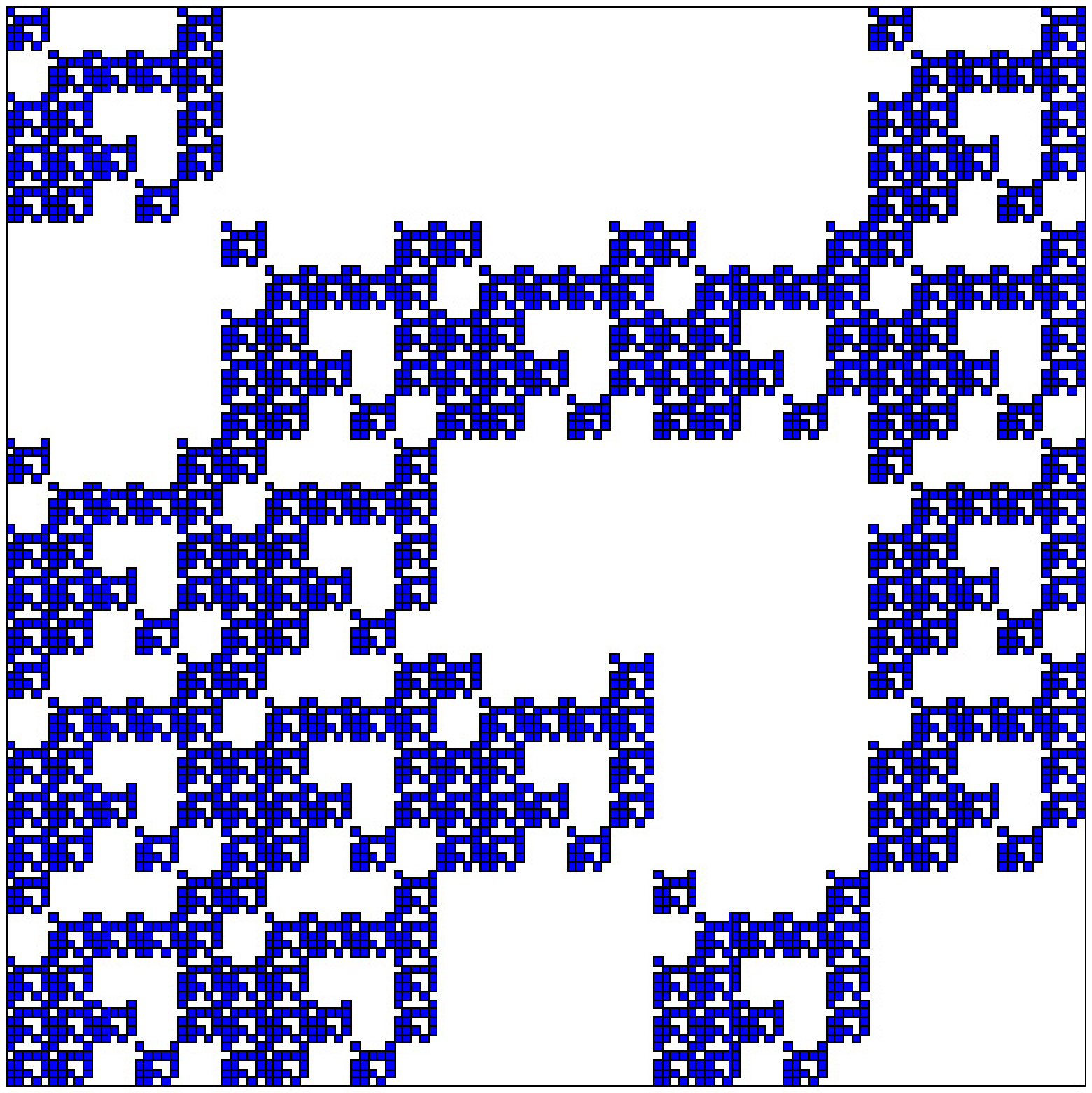}
 }
\caption{}\label{fig.ex}
\end{figure}

\medskip

Clearly in  ${\mathcal G}_0 $, $v_2,v_5,v_8$ are in the same
equivalence class, and $v_4,v_6,v_7$ are in another equivalence
class. Let ${\mathcal V}_0^{0*} = \{\varepsilon_0, v_1, [v_2], v_3,
[v_4], v_9\}$, and from the reduced graph ${\mathcal G}_0^{0*}$  we
obtain $Q_1=Q_0\cup \{\pm (1,1)\}.$

\medskip

In ${\mathcal G}_1$,  we check that $[v_1]=[v_2];\ [v_3]=[v_4]$; and
$[v_3]=[v_9]$.  Then ${\mathcal V}_1^{0*} = \{\varepsilon_0, [v_1],
[v_3] \}$, and from the reduced graph ${\mathcal G}_1^{0*}$, we show
that $Q_2=Q_1\cup\{\pm (2,1)\}.$

\medskip

In ${\mathcal G}_2$,  there is no new reduction on the equivalence
class and we use the same vertex set ${\mathcal V}_2^{0*} =
{\mathcal V}_1^{0*}$, and by checking the reduced graph ${\mathcal
G}_2^{0*}$,  we have $Q_3=Q_2\cup\{\pm (1,2)\}.$

\medskip

Now in  ${\mathcal G}_3 $, we obtain $[v_1]=[v_3]$, so that
${\mathcal V}_3^{0*} = \{\varepsilon_0, [v_1]\}$. Also we have from
the above, there is already an edge $([v_1], [v_3]; e_2) \in
{\mathcal G}_1^*$. This leads to a non-zero loop $\{([v_1], [v_1];
e_2)\}$ in ${\mathcal G}_3^*$. Therefore  ${\mathcal G}_3$ has a
non-zero loop, and the components of $H^c$ are unbounded. On the
other hand, it is easy to see that $\Omega_1=\emptyset$ for any
slope $\tau$, hence there are no line segments in $F$ by Theorem
\ref{th3.3}. Consequently, $F$ is totally disconnected. \hfill
$\Box$}
\end{Exa}

\end{section}

\bigskip

 \noindent {\it Acknowledgements}: The authors would
like to thank Professor Huo-Jun  Ruan for some inspiring
discussions. They are also grateful to the referee for the valuable
comments and suggestions, in particular, for pointing out the two
latest papers of Taylor et al \cite{Tay1, Tay2}.

\end{document}